\documentclass[11pt,reqno]{amsart}

\usepackage[ansinew]{inputenc}

\usepackage{textcomp}
\usepackage{graphicx}

\usepackage{amsthm}
\usepackage{amssymb}
\usepackage{amsfonts}
\usepackage{amsmath}

\newtheorem{thm}{Theorem}[section]
\newtheorem{lem}[thm]{Lemma}
\newtheorem{cor}[thm]{Corollary}

\newtheorem{pro}[thm]{Proposition}

\newtheorem*{pro*}{Proposition}

\theoremstyle{remark}
\newtheorem*{rem}{Remark}

\newcommand{\lk}{\left(}
\newcommand{\rk}{\right)}

\newcommand{\R}{\mathbb{R}}

\newcommand{\N}{\mathbb{N}}
\newcommand{\tr}{\textnormal{Tr}}
\newcommand{\La}{\Lambda}

\title{Sharp spectral estimates  in domains of infinite volume}

\date{\today}

\author{Leander Geisinger}
\address{Leander Geisinger \\ Universität Stuttgart \\ Pfaffenwaldring 57 \\ D - 70569 Stuttgart}
\email{geisinger@mathematik.uni-stuttgart.de}

\author{Timo Weidl} 
\address{Timo Weidl \\Universität Stuttgart \\ Pfaffenwaldring 57 \\ D - 70569 Stuttgart}
\email{weidl@mathematik.uni-stuttgart.de}



\begin{document}

\begin{abstract}
We consider the Dirichlet Laplace operator on open, quasi-bounded domains of infinite volume. For such domains semiclassical spectral estimates based on the phase-space volume -- and therefore on the volume of the domain -- must fail. Here we present a method how one can nevertheless prove uniform bounds on eigenvalues and eigenvalue means which are sharp in the semiclassical limit.

We give examples in horn-shaped regions and so-called spiny urchins. Some results are extended to Schr\"odinger operators defined on quasi-bounded domains with Dirichlet boundary conditions.
\end{abstract}

\maketitle

\section{Introduction} 

Let $V(x)$ be a non-negative function on an open set $\Omega \subset \R^d$, $d \geq 1$. In this article we study the negative spectrum of Schr\"odinger operators
$$
H_\Omega \, = \, - \Delta - V
$$
defined in $L^2(\Omega)$ with Dirichlet conditions on the boundary of $\Omega$. More precisely, one defines $H_\Omega$ to be the self-adjoint operator generated by the quadratic form 
$$
\langle u , H_\Omega u \rangle \, = \, \int_\Omega |\nabla u(x) |^2 \, dx - \int_\Omega V(x) \, |u(x)|^2 \, dx \, ,
$$
with form domain $H_0^1(\Omega)$, see \cite{BirSol87} for details. We always assume that $H_\Omega$ has purely discrete spectrum. Then the negative spectrum of $H_\Omega$ consists of finitely many eigenvalues $-\lambda_1 \leq -\lambda_2 \leq \dots -\lambda_N < 0$, $N < \infty$, counted with multiplicity.
In general, these eigenvalues cannot be calculated explicitly and for large $N$ it is difficult to approximate them numerically. Hence, to deduce information about the eigenvalues one studies also the Riesz means
$$
R_\sigma(V;\Omega) \, = \, \tr(H_\Omega)_-^\sigma \, = \, \sum_{k=1}^N \lambda_k^\sigma
$$
of order $\sigma \geq 0$ and their dependence on $\Omega$ and $V$.

The first rigorous step in this direction dates back to H. Weyl, R. Courant and D. Hilbert \cite{Weyl12a,CouHil24} who calculated the semiclasscial limit of the eigenvalues in the case of a constant potential. To state the general result let us introduce a scaling parameter $\lambda > 0$ and replace the potential $V$ by $\lambda V$. Then for $\sigma \geq 0$ and $V \in L^{\sigma+d/2}(\Omega)$ the limit
\begin{equation}
\label{eq:int:sc1}
R_\sigma(\lambda V;\Omega) \, = \, L^{cl}_{\sigma,d} \int_\Omega V(x) \, dx \, \lambda^{\sigma+d/2} + o(\lambda^{\sigma+d/2}) \, , \quad \lambda \to \infty \, ,
\end{equation}
holds with the semiclassical constant
$$
L^{cl}_{\sigma,d} \, = \, \frac{\Gamma(\sigma+1)}{(4\pi)^{d/2} \Gamma(\sigma+\frac d2 +1)} \, ,
$$
see e.g. \cite{ReeSim78}. To get information about finite potentials one needs to supplement this asymptotic result with uniform estimates. In \cite{LieThi76} it was shown that for $\Omega = \R^d$ and $\sigma > \max\{0,1-d/2\}$ the estimate
$$
R_\sigma(V;\R^d) \, \leq \, L_{\sigma,d} \int_{\R^d} V(x)^{\sigma+d/2} \, dx
$$
holds with certain positive constants $L_{\sigma,d}$. These inequalities have many important applications, for example, in proving the stability of matter \cite{Lieb97,LieSei10}.

Finding the best constants for which the Lieb-Thirring inequalities hold, poses a substantial mathematical challenge. In \cite{LapWei00} the inequalities were established for $\sigma \geq 3/2$ with the sharp constants $L_{\sigma,d} = L^{cl}_{\sigma,d}$. This result immediately implies that for any open set $\Omega \subset \R^d$, $\sigma \geq 3/2$, and any non-negative potential $V \in L^{\sigma + d/2}(\Omega)$
\begin{equation}
\label{eq:int:lt}
R_\sigma(V;\Omega) \, \leq \, L_{\sigma,d}^{cl} \int_\Omega V(x)^{\sigma+d/2} \, dx \, .
\end{equation}
If $V \in L^{\sigma + d/2}(\Omega)$ then both (\ref{eq:int:sc1}) and (\ref{eq:int:lt}) hold and we see that  the bound (\ref{eq:int:lt}) is sharp: It shows the correct power of $V$ and holds with the sharp constant.

In this article we are interested in the case $V \notin L^{\sigma + d/2}(\Omega)$, where the bound (\ref{eq:int:lt}) and even the asymptotics  (\ref{eq:int:sc1}) must fail and one needs to find a new approach to get sharp uniform bounds on eigenvalues means. If $V \notin L^{\sigma+d/2}(\Omega)$ the leading order of the semiclassical limit depends on the potential $V$ and on the geometry of $\Omega$ and it is challenging to find estimates that take these dependencies into account. 

Let us discuss the case of constant potential $V \equiv \La > 0$ on $\Omega$ in more detail. If $\Omega$ is bounded then the semiclassical limit (\ref{eq:int:sc1}) reads as
\begin{equation}
\label{eq:int:sc2}
R_\sigma(\La;\Omega) \, = \, L^{cl}_{\sigma,d} \, |\Omega| \, \Lambda^{\sigma+d/2} + o(\Lambda^{\sigma+d/2}) \, , \quad \sigma \geq 0 \, , \quad \La \to \infty \, ,
\end{equation}
where $|\Omega|$ denotes the volume of $\Omega$. In this case the asymptotic results are supplemented by the Berezin-Lieb-Li-Yau inequality \cite{Berezi72b,Lieb73,LiYau83}: For $\sigma \geq 1$ 
\begin{equation}
\label{eq:int:ber}
R_\sigma(\La;\Omega) \, \leq \, L^{cl}_{\sigma,d} \, |\Omega| \, \Lambda^{\sigma+d/2} \, , \quad \La > 0 \, .
\end{equation}
Again, the constant in this inequality is sharp and cannot be improved. However, under certain conditions on the geometry of $\Omega$ a negative second term exists in the semiclassical limit (\ref{eq:int:sc2}), see \cite{Ivrii80,Hoerma85,SafVas97,Ivrii98,FraGei10}, and the question arises whether (\ref{eq:int:ber}) can be improved by an additional negative correction term. Recently, several results have been found giving answer to this question \cite{FrLiUe02,Melas03,Weidl08,KoVuWe09,GeiWei10,GeLaWe11}. In \cite{FrLiUe02} the corresponding sharp estimate for the discrete Laplacian was improved by a negative remainder term capturing the properties of the second term of the semiclassical limit. The first uniform improvement for the continuous Laplacian is due to A. Melás \cite{Melas03}. He improved the estimate (\ref{eq:int:ber}) for $\sigma \geq 1$, however, the remainder does not reflect the correct order of the second term of the semiclassical limit.

In \cite{Weidl08} this was improved in the case $\sigma \geq 3/2$. Using an inductive argument based on operator-valued Lieb-Thirring inequalities \cite{LapWei00} the Berezin inequality (\ref{eq:int:ber}) was strengthened by a negative remainder term of correct order in comparison with the second term of the semiclassical limit. Here we are not concerned with the remainder term but we apply the same inductive argument to derive sharp spectral inequalities in domains of infinite volume.

However, for unbounded domains $\Omega$ even the discreteness of the spectrum of the Dirichlet Laplacian  is no longer guaranteed. A necessary condition is the so called quasi-boundedness of $\Omega$ \cite{AdaFou03} which is satisfied, by definition, if
\begin{equation*}
\lim_{\begin{subarray}{c} x \in \Omega \\ |x| \to \infty \end{subarray}} \mbox{dist}(x,\, \partial \Omega )  \, = \, 0 \, .
\end{equation*} 
But even for quasi-bounded domains (\ref{eq:int:sc2})  and (\ref{eq:int:ber}) must fail if the volume of $\Omega$ is infinite. In this article we show that one can nevertheless prove uniform bounds on the eigenvalue means for certain domains with infinite volume. In this case the leading order of the semiclassical limit depends on the geometry of $\Omega$, see e.g. \cite{Flecki78,Simon83}. However, applying the induction-in-the-dimension argument from \cite{Weidl08} we can prove sharp estimates valid for all $\La > 0$ that capture the correct asymptotic behavior.

If the potential $V$ is not constant the situation is more difficult. The same inductive argument can still be used to reduce the problem to one dimension. But in contrast to the case of constant potential the eigenvalues of the resulting one-dimensional operator cannot be calculated explicitly. Therefore we have to study the one-dimensional problem in more detail. In particular, we have to analyze the effect of different boundary conditions on the eigenvalues. The result yields an improved version of the semiclassical bound (\ref{eq:int:lt}). Again, this sharp Lieb-Thirring inequality with remainder term can be applied in situations, where all known results -- in particular (\ref{eq:int:sc1}) and (\ref{eq:int:lt}) -- fail.

The remainder of the article is structured as follows. First we mention some key ingredients of the proofs. In particular, we review the induction-in-the-dimension argument form \cite{Weidl08} and adapt it to our needs here. This is done in Section \ref{sec:ind}. 

In Section \ref{sec:const} we consider constant potentials on domains with infinite volume. We give examples, where the leading order of the semiclassical limit depends on the geometry of the domain $\Omega$. In this situation we derive sharp upper bounds on the eigenvalue means.

The last part of the article is devoted to the general setting of non-constant potentials. In Section \ref{sec:schr} we first analyze the effect of different boundary conditions on the eigenvalues of one-dimensional Schr\"odinger operators. We find an improvement of (\ref{eq:int:lt}) that can be generalized to higher dimensions. Finally, we give an example to show that the result applies for certain potentials $V \notin L^{\sigma+d/2}(\Omega)$.


\section{Induction in the dimension} 
\label{sec:ind}

In this section we prove an inequality reducing estimates for eigenvalue means of the operator $H_\Omega$ to estimates for one-dimensional Schr\"odinger operators. The proof relies on a lifting technique from \cite{Laptev97} and uses operator-valued Lieb-Thirring inequalities \cite{LapWei00}.
Here we follow the proof from \cite{Weidl08}, where this approach of induction-in-the-dimension is employed to derive improvements of (\ref{eq:int:ber}) for constant potentials.

Fix a Cartesian coordinate system in $\mathbb{R}^d$ and for $x \in \R^d$ write $x = (x',t) \in \mathbb{R}^{d-1} \times \mathbb{R}$. For $x' \in \R^{d-1}$ consider one-dimensional sections $\Omega(x') = \{t\in \R \, : \, (x',t) \in \Omega\}$. If not empty, each section $\Omega(x')$ consists of at most countably many open intervals $J_k(x') \subset \R$, $k = 1, \dots, N(x') \leq \infty$.

For $x = (x',t) \in \Omega$ put $V_{x'}(t) \, = \, V(x)$ and let the one-dimensional Schr\"odinger operators
$$
H_k(x') \, = \, - \frac{d^2}{dt^2} - V_{x'} \, , \quad k = 1, \dots, N(x') \, ,
$$
be defined in  $L^2(J_k(x'))$ with Dirichlet boundary conditions. Finally let 
\begin{equation}
\label{eq:ind:w}
W(x',V) \, = \, \bigoplus_{k=1}^{N(x')} H_k(x')_-
\end{equation}
be the negative part of the Schr\"odinger operator $-d^2/dt^2 - V_{x'}$ given on $\Omega(x')$ with Dirichlet boundary conditions at the endpoints of each interval forming $\Omega(x')$, that is, on the boundary of $\Omega(x')$.

Using operator-valued Lieb-Thirring inequalities one can estimate eigenvalue means of $H_\Omega$ in terms of $W(x',V)$.

\begin{pro}
\label{pro:basic}
For $\sigma \geq 3/2$ we have
$$
R_\sigma(V;\Omega) \, \leq \, L^{cl}_{\sigma,d-1} \int_{\R^{d-1}} \tr W(x',V)^{\sigma + (d-1)/2} \, dx' \, .
$$
\end{pro}

\begin{rem}
In the case of constant potential $V \equiv \La > 0$ the trace of $W(x',\La)$ can be evaluated explicitly. If $\Omega$ is bounded, a detailed analysis of the resulting estimate leads to improved Berezin-Li-Yau inequalities with a remainder term capturing the properties of the second term of the semiclassical limit \cite{Weidl08, GeiWei10,GeLaWe11}.
\end{rem}

\begin{proof}[Proof of Proposition \ref{pro:basic}]
We consider the quadratic form $\langle u, H_\Omega u \rangle$ and evaluate it on functions $u$ from the form core $C_0^\infty(\Omega)$. We get
\begin{align*}
\langle u, H_\Omega u \rangle_{L^2(\Omega)} &= \left\| \nabla u \right\|^2_{L^2(\Omega)} - \int_\Omega V |u|^2 \, dx\\
&= \left\| \nabla' u \right\|^2_{L^2(\Omega)} + \int_{\R^{d-1}} dx' \int_{\Omega(x')} \lk \left| \partial_t u(x',t) \right|^2 - V_{x'} (t)\left| u(x',t) \right|^2 \rk dt \, ,
\end{align*}
where $\nabla'$ denotes the gradient in the first $(d-1)$-coordinates.

For fixed $x' \in \R^{d-1}$ the functions $u(x',\cdot)$ belong to $C_0^\infty(\Omega(x'))$ and therefore to the form core of $W(x',V)$. It follows that 
\begin{equation}
\label{eq:ind:qform}
\langle u, H_\Omega u \rangle_{L^2(\Omega)} \, \geq \,  \left\| \nabla' u \right\|^2_{L^2(\Omega)} - \int_{\R^{d-1}} \langle u(x',\cdot), W(x',V) u(x',\cdot)\rangle_{L^2(\Omega(x'))} \, dx' \, .
\end{equation}
To apply operator-valued Lieb-Thirring inequalities we need to extend these forms to $\R^d$. More precisely, we extend both sides of (\ref{eq:ind:qform}) by zero to $C_0^\infty(\R^d \setminus \partial \Omega)$ which is a form core of $(-\Delta_{\R^d \setminus \Omega}) \oplus H_\Omega$. This operator corresponds to the left-hand side of (\ref{eq:ind:qform}), while the semi-bounded form on the right-hand side is closed on the larger domain $H^1(\R^{d-1} , L^2(\R))$, where it corresponds to the operator
\begin{equation}
\label{eq:ind:op}
-\Delta' \otimes \mathbb{I} - W(x',V)
\end{equation}
defined in $L^2(\R^{d-1},L^2(\R))$. Due to the positivity of $(-\Delta_{\R^d \setminus \Omega})$ the variational principle implies
\begin{equation} 
\label{eq:ind:trace1}
R_\sigma(V;\Omega) \, = \, \tr \lk -\Delta_{\R^d \setminus \Omega} \oplus H_\Omega \rk_-^\sigma \, \leq \, \tr \lk - \Delta' \otimes \mathbb{I} - W(x',V) \rk_-^\sigma \, .
\end{equation}
Now we apply sharp Lieb-Thirring inequalities \cite{LapWei00} to the Schr\"{o}dinger operator (\ref{eq:ind:op}) with operator-valued potential $W(x',V)$. For $\sigma \geq  3/2$ we obtain
\begin{equation}
\label{eq:ind:trace2}
\tr \left( -\Delta' \otimes \mathbb{I} - W(x',V) \right)_-^\sigma \, \leq \,  L^{cl}_{\sigma,d-1} \int_{\mathbb{R}^{d-1}} \textnormal{Tr} W(x',V)^{\sigma+(d-1)/2} \, dx' 
\end{equation}
and the claim follows from (\ref{eq:ind:trace1}) and (\ref{eq:ind:trace2}).
\end{proof}


\section{Constant potentials}
\label{sec:const}

In this section we assume $V \equiv \La > 0$ on quasi-bounded open sets $\Omega \subset \R^d$, $d \geq 2$. First we remark the following relations between the eigenvalue means. For $0 \leq \gamma < \sigma$ we have \cite{AizLie78}
\begin{equation}
\label{eq:ind:al}
R_\sigma(\La;\Omega) \, = \, \frac{1}{B(\gamma+1,\sigma-\gamma)} \int_0^\infty \tau^{\sigma-\gamma-1} R_\gamma((\La-\tau)_+;\Omega) \, d\tau \, ,
\end{equation}
where $B$ denotes the Beta-function. Hence one can use bounds or asymptotic results for $R_\gamma$ to deduce the respective results for $R_\sigma$ with $\sigma > \gamma \geq 0$. 
Conclusions from eigenvalue means of higher order to means of lower order are more cumbersome since eigenvalue means of lower order are less smooth. 
To derive uniform bounds on the counting function, that is, on $R_0(\La;\Omega)$ one can make use of the estimate \cite{Laptev97}
\begin{equation}
\label{eq:ind:lap}
R_0(\La;\Omega) \, \leq \, (\tau \La)^{-\sigma} R_\sigma((1+\tau)\La;\Omega) \, ,  \quad \tau >0 \, , \quad \La > 0 \, , \quad \sigma > 0 \, ,
\end{equation}
and optimize the right hand side in $\tau > 0$. In general, sharp constants are lost but usually the correct order of growth in $\La$ is preserved.

In the following we consider specific domains with infinite volume. The discreteness of the spectrum of the Dirichlet Laplace operator defined on these domains can be deduced from the following sufficient condition \cite{Adams70}.

\begin{lem}
\label{lem:adams}
Let $\Omega$ be an open subset of $\R^d$ and let $Q$ be a cube with sides parallel to the coordinate axes. Let $\mu_{d-1}(Q,\Omega)$ denote  the maximum of the $(d-1)$-dimensional measure of $P(Q \setminus \Omega)$, where the maximum is taken over all projections $P$ onto $(d-1)$-dimensional faces of $Q$. 

Assume that for every $\epsilon > 0$ there exist $h \leq 1$ and $r \geq 0$ such that for every cube $Q$ of side length $h$ with sides parallel to the coordinate axes and with $Q \cap \{ x \in \R^d \, : \, |x| > r \} \, \neq \, \emptyset$ we have
$$
\frac{\mu_{d-1} (Q,\Omega) }{h^{d+1}} \, \geq \, \frac 1\epsilon \, .
$$

Then the embedding $H_0^1(\Omega) \hookrightarrow L^2(\Omega)$ is compact.
\end{lem}

In the following examples the trace of the operator $W(x',\Lambda)$ given in (\ref{eq:ind:w}) can be calculated explicitly and we find that Proposition \ref{pro:basic} yields sharp estimates on eigenvalue means.

\subsection{Horn-shaped regions} 
\label{sec:horn}

First we consider  horn-shaped regions, domains stretching to infinity along distinguished directions, see \cite{Berg92a} for a general definition. These regions turn out to be of interest in different situations, see e.g. \cite{Simon83,Berg84b,DavSim92,Berg92a,MatMue06,Lundho10}. They were introduced in \cite{Simon83}, where the semiclassical limit of the counting function was calculated for domains 
\begin{equation}
\label{eq:horn:simon}
\Omega_\nu \, = \, \left\{ (x,y) \in \R^2 \, : \, |x| \cdot |y|^\nu \, \leq \, 1 \right\} \, , \quad \nu \geq 1 \, ,
\end{equation}
see Figure 1. Note that discreteness of the spectrum of $H_{\Omega_\nu}$ can be deduced from Lemma \ref{lem:adams}.

\begin{figure}[ht]
\label{fig:horn}
\centering
\includegraphics[width=9cm]{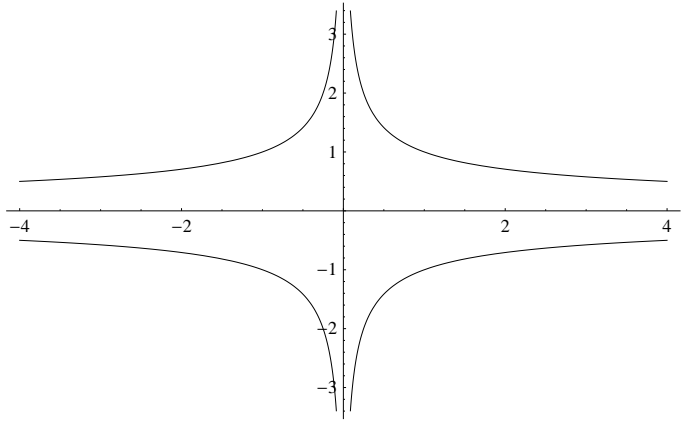}     
\caption{The set $\Omega_2$.}
\end{figure}

In \cite{GeiWei10} it was shown that the methods introduced in Section \ref{sec:ind} yield sharp upper bounds on the trace of the heat kernel of the Dirichlet Laplacian on various horn-shaped regions. Here we derive sharp bounds on  eigenvalue means and order-sharp bounds on the counting function.

Let us recall the following asymptotic results from \cite{Simon83}.
For $\nu > 1$ the limit 
$$
R_0(\La;\Omega_\nu) \, = \, \zeta(\nu) \, \lk \frac 2\pi \rk^\nu \frac{\Gamma\lk \frac \nu2 + 1 \rk  }{\sqrt{\pi}\,\Gamma \lk \frac{\nu+3}{2} \rk} \, \La^{(\nu+1)/2} + o \lk \La^{(\nu+1)/2} \rk \, , \quad \La \to \infty \, ,
$$
holds, where $\zeta(\nu)$ denotes the Zeta function. Moreover, for $\nu = 1$ 
$$
R_0(\La;\Omega_1) \, = \, \frac 1\pi \, \La \ln \La + o \lk \La \ln \La  \rk \, , \quad \La \to \infty \, .
$$
Applying (\ref{eq:ind:al}) with $\gamma = 0$ we obtain for $\sigma > 0$ and $\nu > 1$
\begin{equation}
\label{eq:horn:asympt1}
R_\sigma(\La;\Omega_\nu) \, = \, \zeta(\nu) \, \lk \frac 2\pi \rk^\nu \frac{B( \frac \nu 2+1,\sigma+1)}{B \lk \sigma+\frac{\nu+3}{2}, \frac 12 \rk} \, \La^{\sigma+(\nu+1)/2} + o \lk \La^{\sigma+(\nu+1)/2} \rk \, , \quad \La \to \infty \, ,
\end{equation}
and for $\nu = 1$
\begin{equation}
\label{eq:horn:asympt2}
R_\sigma(\La;\Omega_1) \, = \, \frac{1}{\pi \, (\sigma +1)} \, \La^{\sigma+1} \, \ln \La + o \lk \La^{\sigma+1} \, \ln \La \rk \, , \quad \La \to \infty \, .
\end{equation}

In order to treat domains in higher dimensions we generalize the notions from \cite{Simon83} and put
$$
\Omega_\nu \, = \, \left\{ (x',x_d) \in \R^{d-1} \times \R \, : \, |x'| \cdot |x_d|^{\nu/(d-1)} \, \leq \, 1 \right\} \, , \quad d \geq 2 \, , \quad \nu > 1 \, .
$$
For these domains of infinite volume an application of Proposition \ref{pro:basic} yields sharp spectral estimates.

\begin{thm}
\label{thm:horn:nu}
For $\sigma \geq 3/2$, $\nu > 1$, and all $\La > 0$ the estimate
$$
R_\sigma(\La ; \Omega_\nu) \, \leq \, \frac{\zeta(\nu)}{2^{d-1}(d-1)}  \lk \frac{2}{\pi} \rk^\nu  \frac{\Gamma(\frac \nu 2 +1) \Gamma (\sigma+1)}{\Gamma (\frac{d+1}{2}) \Gamma \lk \sigma + \frac{d+1+\nu}{2} \rk} \, \La^{\sigma+(d-1+\nu)/2} 
$$
holds.
\end{thm}

\begin{rem}
For $d = 2$ we conclude that the bound
$$
R_\sigma(\La;\Omega_\nu) \, \leq \, \zeta(\nu) \, \lk \frac 2\pi \rk^\nu \frac{B(\frac \nu 2+1,\sigma+1)}{B \lk \sigma+\frac{\nu+3}{2},\frac 12 \rk} \, \La^{\sigma+(\nu+1)/2} 
$$
holds for $\sigma \geq 3/2$ and all $\La >0$. Comparing this bound with the asymptotic relation (\ref{eq:horn:asympt1}) we see that the estimate is sharp: For horn-shaped regions, just as well as for bounded domains, the leading term of the semiclassical limit  yields a uniform upper bound.
\end{rem}

\begin{proof}[Proof of Theorem \ref{thm:horn:nu}]
In this setting the section $\Omega_\nu(x')$ consists of one open interval 
$$
(-|x'|^{(1-d)/\nu}, |x'|^{(1-d)/\nu}) \, .
$$
Since $V \equiv \La$, the trace of the operator-valued potential $W(x',\La)$ defined in (\ref{eq:ind:w}) can be evaluated explicitly. We find
$$
\tr W(x',\Lambda) \, = \, \sum_{j \in \N} \lk \La - \frac{\pi^2 j^2}{4 |x'|^{2(1-d)/\nu}} \rk_+ \, .
$$
Applying Proposition \ref{pro:basic} yields
\begin{align*}
R_\sigma(\Lambda;\Omega_\nu) \, &\leq \, L^{cl}_{\sigma,d-1} \int_{\R^{d-1}} \sum_{j \in \N} \lk \La - \frac{\pi^2  j^2}{4 |x'|^{2(1-d)/\nu}} \rk_+^{\sigma+(d-1)/2} dx' \, \La^{\sigma+(d-1)/2} \\
&= \, L^{cl}_{\sigma,d-1} \omega_{d-1} \sum_{j \in \N} \int_0^\infty \lk 1 - \frac{\pi^2  j^2}{4 \La r^{2(1-d)/\nu}} \rk_+^{\sigma+(d-1)/2} r^{d-2} dr \, \La^{\sigma+(d-1)/2}  \\
&= \, L^{cl}_{\sigma,d-1}  \omega_{d-1} \zeta(\nu) \lk \frac{2}{\pi} \rk^\nu \frac{\nu B\lk \sigma + \frac{d+1}{2}, \frac \nu 2 \rk}{2(d-1)} \, \La^{\sigma+(d-1+\nu)/2}\, ,
\end{align*}
where $\omega_{d-1}$ denotes the volume of the unit sphere in $\R^{d-1}$.
We insert the identity
$$
L^{cl}_{\sigma,d-1} \, \omega_{d-1} \, \frac{\nu  B\lk \sigma + \frac{d+1}{2}, \frac \nu 2 \rk }{2(d-1)}  \, = \, \frac{\Gamma(\sigma +1) \Gamma(\frac \nu 2 +1)}{2^{d-1} (d-1) \Gamma(\frac{d+1}{2}) \Gamma(\sigma+ \frac{d+1+\nu}{2})} 
$$
and arrive at the claimed estimate.
\end{proof}

Now we apply (\ref{eq:ind:lap}) to deduce order-sharp bounds on the counting function. 

\begin{cor}
\label{cor:horn1}
For $\nu > 1$ and all $\Lambda > 0$ the estimate
$$
R_0(\La;\Omega_\nu) \, \leq \, C_{d,\nu} \,  \La^{(d-1+\nu)/2}
$$
holds with a constant
$$
C_{d,\nu} \, \leq \, \frac{(d+\nu+2)^{(d+\nu+2)/2}}{3^{3/2} (d+\nu-1)^{(d+\nu-1)/2}} \frac{\zeta(\nu)}{2^{d-1}(d-1)}  \lk \frac{2}{\pi} \rk^\nu  \frac{\Gamma(\frac \nu 2 +1) \Gamma (\frac 52)}{\Gamma (\frac{d+1}{2}) \Gamma \lk \frac{d+\nu}{2} +2 \rk} \, .
$$
\end{cor}

\begin{proof}
We use (\ref{eq:ind:lap}) with $\sigma = 3/2$ and apply Theorem \ref{thm:horn:nu} to obtain
\begin{align*}
R_0(\La;\Omega_\nu) \, &\leq \, \frac{1}{(\tau \La)^{3/2}}  \frac{\zeta(\nu) }{2^{d-1}(d-1)} \lk \frac{2}{\pi} \rk^\nu  \frac{\Gamma(\frac \nu 2 +1) \Gamma (\frac 52)}{\Gamma (\frac{d+1}{2}) \Gamma \lk \frac{d+\nu}{2} +2 \rk} \, ((1+\tau)\La)^{(d+\nu)/2+1} \\
&\leq \, \frac{(1+\tau)^{(d+\nu)/2+1}}{\tau^{3/2}}  \frac{\zeta(\nu) }{2^{d-1}(d-1)} \lk \frac{2}{\pi} \rk^\nu  \frac{\Gamma(\frac \nu 2 +1) \Gamma (\frac 52)}{\Gamma (\frac{d+1}{2}) \Gamma \lk \frac{d+\nu}{2} +2 \rk} \, \La^{(d-1+\nu)/2} \, .
\end{align*}
Minimizing in $\tau > 0$ yields $\tau_{\min} = 3/(d+\nu-1)$ and inserting this we obtain the claimed result. 
\end{proof}

Let us now consider the critical case $\nu = 1$ in dimension $d = 2$. Here the domain yields two equally strong singularities and we cannot distinguish one direction. However, choosing an intermediate direction we obtain a sharp estimate with a remainder term.

\begin{thm}
\label{thm:horn:1}
Let $\sigma \geq 3/2$. Then for $\La \leq \pi^2/16$ we have $R_\sigma(\Lambda;\Omega_1) = 0$ and for $\Lambda > \pi^2/16$ the estimate 
$$
R_\sigma(\La ; \Omega_1) \, \leq \, \frac{1}{\pi \, (\sigma +1)}  \La^{\sigma+1} \ln \La + \frac{C}{\sigma+1}  \La^{\sigma+1} 
$$
holds with a constant
$$
C \, < \, \frac{33+16 \ln(\frac 4 \pi)}{8 \pi} \, < \, 1.47 \, .
$$
\end{thm}

\begin{rem}
Again, comparing this inequality with the asymptotics (\ref{eq:horn:asympt2}), we see that the main term of the bound is sharp.
\end{rem}

\begin{proof}[Proof of Theorem \ref{thm:horn:1}]
Since the function $|\Omega_1(x)| = \frac 1x$ has non-integrable singularities at zero and at infinity we have to choose a coordinate system $(x_1,x_2) \in \R^2$ rotated by $\frac \pi 4$ with respect to the coordinate system $(x,y) \in \R^2$ which was used in definition (\ref{eq:horn:simon}). We get
$$
\Omega_1(x_1) \, = \, \left\{x_2 \in \R \, : \,  |x_2| \, \leq \,\sqrt{|x_1|^2+2} \right\}
$$
for $|x_1| \leq \sqrt{2}$ and
$$
\Omega_1(x_1) \, = \, \left\{x_2 \in \R \, : \, \sqrt{|x_1|^2-2} \, \leq \, |x_2| \, \leq \,\sqrt{|x_1|^2+2} \right\} 
$$
for $|x_1| > \sqrt 2$. To simplify the following calculations and the resulting bound we confine ourselves to rough estimates which are nevertheless sufficient to prove the sharp constant in the leading term. First, we estimate $|\Omega_1(x_1)| \leq 4$ for $|x_1| \leq 2$ and
$$
|\Omega_1(x_1)| \, \leq \, 2 \lk \sqrt{|x_1|^2+2} - \sqrt{|x_1|^2-2} \rk \, \leq \, \frac{4}{|x_1|} + \frac{2}{|x_1|^3}
$$
for $|x_1| > 2$.

Suppose that $\La \leq \pi^2/16$. Since $|\Omega_1(x_1)| \leq 4$ for all $x_1 \in \R$ we get
$$
\tr W(x_1,\La) \, = \,  \sum_{j \in \N} \lk \La - \frac{\pi^2 j^2}{|\Omega(x_1)|^2} \rk_+ \, = \, 0 
$$
for all $x_1 \in \R$. From Proposition \ref{pro:basic} it follows that $R_\sigma(\La;\Omega_1) = 0$ for $\La \leq \pi^2/16$. On the other hand, if $\La > \pi^2 /16$ Proposition \ref{pro:basic} implies
\begin{align}
\nonumber
R_\sigma(\La;\Omega_1) \, & \leq \, L^{cl}_{\sigma,1} \int_\R \sum_{j \in \N} \lk \La - \frac{\pi^2 j^2}{|\Omega(x_1)|^2} \rk_+^{\sigma+1/2} \, dx_1 \\
\label{eq:horn:rotated}
& \leq \, 4 L^{cl}_{\sigma,1}  \sum_{j \in \N} \lk \La - \frac{\pi^2 j^2}{16} \rk_+^{\sigma+1/2} +  2 L^{cl}_{\sigma,1}  \int_2^\infty \sum_{j \in \N} \lk \La -
\frac{\pi^2 j^2}{l(x_1)^2} \rk_+^{\sigma+1/2} \, dx_1 \, ,
\end{align}
with 
$$
l(x_1) \, = \,  \frac{4}{|x_1|} + \frac{2}{|x_1|^3} \, .
$$

Note that for $A > 0$ and $\gamma > 0$ we have
\begin{equation}
\label{eq:horn:beta}
\sum_{j \in \N} \lk 1- \frac{j^2}{A^2} \rk_+^\gamma \, \leq \, \frac A2 B \lk \frac 12 , \gamma +1 \rk \, ,
\end{equation}
thus
\begin{equation}
\label{eq:horn:rot1}
4 L^{cl}_{\sigma,1}  \sum_{j \in \N} \lk \La - \frac{\pi^2 j^2}{16} \rk_+^{\sigma+1/2} \, \leq \, \frac 8\pi  L^{cl}_{\sigma,1} B \lk \frac 12 , \sigma + \frac 32 \rk \La^{\sigma+1} \, = \, \frac 4\pi \frac{1}{\sigma+1}  \La^{\sigma+1} \, .
\end{equation}

Now we turn to the second term in (\ref{eq:horn:rotated}). Put $x(\La) = (4\sqrt{\La})/\pi + \pi/(4\sqrt{\La})$. For $x_1 \geq x(\La)$ we have $l(x_1) \leq \pi / \sqrt{\La}$, hence
$$
\sum_{j \in \N} \lk \La - \frac{\pi^2 j^2}{l(x_1)^2} \rk_+ \, = \, 0 \, .
$$
In view of (\ref{eq:horn:beta}) it follows that
\begin{align}
\nonumber
2 L^{cl}_{\sigma,1} \int_2^\infty \sum_{j \in \N} \lk \La - \frac{\pi^2 j^2}{l(x_1)^2} \rk_+^{\sigma+1/2} \, dx_1 \, &\leq \, \frac 1\pi L^{cl}_{\sigma,1}  B \lk \frac 12, \sigma+\frac 32 \rk \int_2^{x(\La)} l(x_1) dx_1 \, \La^{\sigma+1}\\
\label{eq:horn:rot2}
& = \frac{1}{2\pi} \frac{1}{\sigma+1} \int_2^{x(\La)} l(x_1) dx_1 \, \La^{\sigma+1} \, .
\end{align}
By definition of $x(\La)$ and $l(x_1)$,
\begin{align}
\nonumber
\int_2^{x(\La)} l(x_1) dx_1 \, &= \, \int_2^{x(\La)} \lk \frac{4}{x_1} + \frac{2}{x_1^3} \rk dx_1 \\
\nonumber
&\leq \, 2 \ln \La + 4 \ln \lk \frac 4\pi + \frac{\pi}{4 \La} \rk - 4 \ln 2 + \frac 14 \\
\label{eq:horn:rot3}
&\leq \, 2 \ln \La +4 \ln \lk \frac 4\pi \rk + \frac 14
\end{align}
for $\La > \pi^2/16$. Inserting (\ref{eq:horn:rot1}), (\ref{eq:horn:rot2}) and (\ref{eq:horn:rot3}) into (\ref{eq:horn:rotated}) finishes the proof. 
\end{proof}

Again we can apply (\ref{eq:ind:lap}) to deduce order-sharp bounds on the counting function.

\begin{cor}
For $\La \leq \pi^2/16$ we have $R_0(\La;\Omega_1) = 0$ and for $\La > \pi^2/16$ the estimate
$$
R_0(\La;\Omega_1) \, \leq \, \lk \frac53 \rk^{3/2} \frac 1\pi \, \La \ln \La + C  \La \, ,
$$
holds, with a constant
$$
C \, < \, \, \sqrt{\frac 53} \, \frac{825+400 \ln \lk \frac 4\pi \rk +360 \pi\ln \lk \frac 53 \rk}{72\pi}  \, < \, 8.56 \, .
$$
\end{cor}


\subsection{Spiny urchins}

In this subsection we study the eigenvalues of the Dirichlet Laplacian on so called spiny urchins, radially symmetric domains $\Omega_S \subset \R^2$ with infinite volume, which were introduced in \cite{Clark67}.

To construct $\Omega_S$ we use polar coordinates $(r,\varphi) \in [0, \infty) \times [0, 2 \pi)$ and choose an increasing sequence $(r_n)_{n \in \N}$ of positive real numbers and put $r_0 = 0$. For $n \in \mathbb{N}_0$ and $k = 1, 2, \dots, 2^{n+2}$ let
\begin{equation*}
\Gamma_{n,k} \, = \, \left\{ (r,\varphi) \,  : \, r \geq r_n \, , \, \varphi = \frac{k-1}{2^{n+1} } \pi \right\}
\end{equation*}
be semi-axes and define 
\begin{equation*}
\Omega_S \, = \, \R^2 \setminus \bigcup_{n,k} \Gamma_{n,k} \, ,
\end{equation*}
see Figure 2. Note that this domain, though quasi bounded, has empty exterior. However, if 
\begin{equation}
\label{eq:spindisc}
\lim_{n \to \infty} r_n \, 2^{-n} \, = \, 0 \, ,
\end{equation}
then discreteness of the spectrum of $H_{\Omega_S}$ can be deduced from Lemma \ref{lem:adams}, see also \cite{Berg92b}.

\begin{figure}[ht]
\label{fig:sonne}
\centering
\includegraphics[width=9cm]{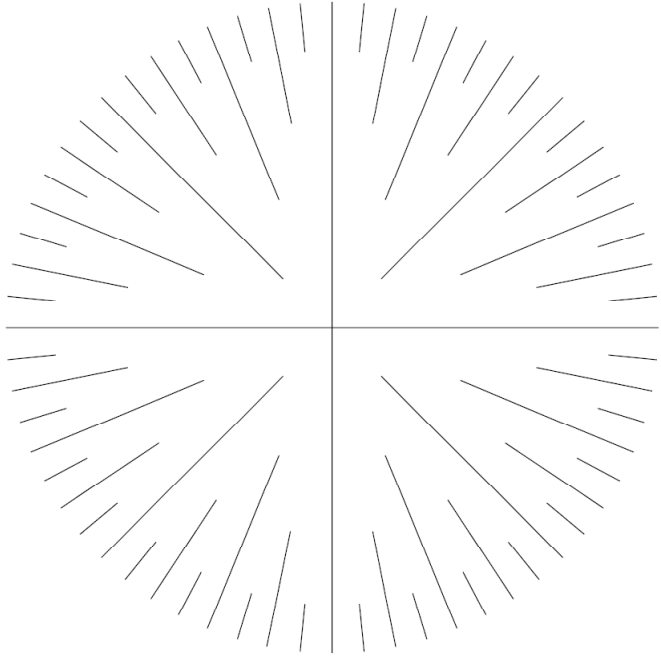}     
\caption{The set $\Omega_S$.}
\end{figure}

For $r_n = n$ the domain $\Omega_S$ was analyzed in \cite{Flecki78}, where the leading term of the semiclassical limit was calculated: For $r_n =n$ the asymptotic relation
$$
R_0(\La;\Omega_S) \, = \, C \, \La (\ln \La)^2 + o \lk \La (\ln \La)^2 \rk \, , \quad \La \to \infty \, ,
$$
holds with a constant $C >0$.

The general setting of an arbitrary increasing sequence $(r_n)_{n \in \N_0}$ was studied in \cite{Berg92b}: If $r_0 > 0$ and (\ref{eq:spindisc}) is satisfied then for all $\La > 2^{14} r_0^{-2}$ the bound
$$
R_0(\La;\Omega_S) \, \leq \, 50(8^{-1}+8\pi)^2 \, \La \, r^2_{K(\La)}
$$
holds with $K(\La)  =  \max\{ n \in \N \, : \, r_n 2^{-n}  > (32)^{-1} \sqrt \La \}$. Moreover, there is a similar lower bound.
Here we extend and improve the upper bound: We derive order-sharp estimates on the eigenvalue means of $H_{\Omega_S}$ valid for all $\La > 0$. 

First, we need to adapt Proposition \ref{pro:basic} to the radially symmetric situation. For $r \in (0,\infty)$ put
$$
\Omega_S(r) \, = \, \left\{ \varphi \in [0,2\pi) \, : \, (r,\varphi) \in \Omega_S \right\} \, .
$$
Then $\Omega_S(r)$ consists of finitely many open intervals $I_k(r)$, $k = 1, \dots, N(r)$. Choose $u \in C_0^\infty(\Omega_S)$ and consider the quadratic form
\begin{align}
\nonumber
\langle u, H_{\Omega_S} u \rangle_{L^2(\Omega_S)} \, &= \, \int_\Omega \overline{u(x)} \lk -\Delta u(x) - \Lambda u(x) \rk dx\\
\label{eq:spin:form}
&= \, \int_0^\infty \int_{\Omega(r)} \overline{u(r,\varphi)} \lk -\partial_r^2 - \frac 1r \partial_r - \frac{1}{r^2} \partial^2_\varphi - \Lambda \rk u(r,\varphi) \, d\varphi \, r \,  dr \, .
\end{align}
For fixed $r > 0$ the function $u_r(\varphi) = u(r, \varphi)$ belongs to $C_0^\infty(\Omega_S(r))$. It satisfies Dirichlet boundary conditions at the endpoints of the intervals $I_k(r)$, $k = 1, \dots, N(r)$. 

To rewrite the form in the ground state representation put $v(r,\varphi) = \sqrt{r} \, u(r,\varphi)$. Then again $v(r,\varphi)$ belongs to $C_0^\infty(\Omega_S)$ and for fixed $r > 0$, we have $v_r(\varphi) = v(r,\varphi) \in C_0^\infty(\Omega_S(r))$. Moreover,
$$
\int_0^\infty \int_{\Omega(r)} |u(r,\varphi)|^2 d\varphi \, r \, dr \, = \, \int_0^\infty \int_{\Omega(r)} |v(r,\varphi)|^2 d\varphi  \, dr 
$$
and
$$
\lk  -\partial_r^2 - \frac 1r \partial_r - \frac{1}{r^2} \partial^2_\varphi\rk u(r,\varphi) \, = \,  \frac{1}{\sqrt r} \lk  - \partial_r^2 - \frac{1}{4r^2}  - \frac{1}{r^2} \partial^2_\varphi \rk v(r,\varphi) \, .
$$
Inserting this into (\ref{eq:spin:form}) we obtain
\begin{equation}
\label{eq:spin:form2}
\langle u, H_{\Omega_S} u \rangle_{L^2(\Omega_S)} \, = \, \int_0^\infty \int_{\Omega_S(r)} \lk |\partial_r v|^2 + \frac{1}{r^2} |\partial_\varphi v|^2 - \lk \frac{1}{4r^2}+\La \rk |v|^2 \rk d\varphi \, dr \, .
\end{equation}

In this setting, we define the Schr\"odinger-type operators
$$
H_k(r) = - \frac{1}{r^2} \frac{d^2}{d\varphi^2} - \lk \frac{1}{4r^2} + \La \rk  \, , \quad k = 1, \dots, N(r) \, ,
$$
in $L^2(I_k(r))$ with Dirichlet boundary conditions at the endpoints of $I_k(r)$. In the same way as in (\ref{eq:ind:w}) let
$$
W(r,\La) \, = \, \bigoplus_{k=1}^{N(r)} H_k(r)_-
$$
be the negative part of the operator
$$
- \frac{1}{r^2} \frac{d^2}{d\varphi^2} - \lk \frac{1}{4r^2}+\La \rk
$$
in $L^2(\Omega_S(r))$ with Dirichlet boundary conditions. In view of (\ref{eq:spin:form2}) we can apply Proposition \ref{pro:basic} and for $\sigma \geq 3/2$ we get
\begin{equation}
\label{eq:spin:basic}
R_\sigma(\La;\Omega_S) \, \leq \, L^{cl}_{\sigma,1} \int_0^\infty \tr W(r,\Lambda)^{\sigma+1/2} dr \, .
\end{equation}

To estimate the right hand side and to derive bounds on the eigenvalues means we assume that (\ref{eq:spindisc}) is satisfied and that
\begin{equation}
\label{eq:spin:small}
r_{n+1} \, \leq \, 2  r_n
\end{equation}
holds for all $n \in \N$.
Then the sequence
$$
\frac{2^{2n}}{r_n^2} - \frac{1}{4 r_n^2} \, , \quad  n \in \mathbb{N} \, ,
$$
is increasing and for all $\La > 15/4 \cdot r_1^{-2}$ there is a unique index $\hat n(\La) \in \N$ satisfying
\begin{equation}
\label{eq:spin:nm}
\La \, > \, \frac{2^{2n}}{r_{n}^2} - \frac{1}{4 r_{n}^2} \ \textnormal{for all} \ n \leq \hat n(\La) \quad \mbox{and} \quad \La \, \leq \, \frac{2^{2n}}{r_n^2} - \frac{1}{4 r_n^2} \ \mbox{for all} \ n > \hat n(\La) \, . 
\end{equation}
To simplify notation we put $\hat r(\La) = r_{\hat n(\La)}$.

\begin{lem}
\label{lem:spin:main}
Let $\sigma \geq 3/2$ and assume that (\ref{eq:spindisc}) and (\ref{eq:spin:small}) are satisfied. Then for $\La \leq 15/4 \cdot r_1^{-2}$ we have $R_\sigma(\La;\Omega_S) = 0$ and for $\La > 15/4 \cdot r_1^{-2}$ the estimate
$$
R_\sigma(\La ; \Omega_S) \, \leq \, L^{cl}_{\sigma,2} \, \pi \hat r(\La)^2 \La^{\sigma+1} + C_\sigma \, \La^{\sigma} \ln \lk \La \hat r(\La) \rk \, .
$$
holds with a constant $C_\sigma > 0$ depending only on $\sigma$.
\end{lem}

\begin{rem}
If we compare the main term of this bound with the Berezin inequality (\ref{eq:int:ber}) we see that the effective domain that enters into the bound is a disk with radius $\hat r(\La)$. 
\end{rem}

\begin{proof}[Proof of Lemma \ref{lem:spin:main}]
In view of (\ref{eq:spin:basic}) we have to estimate
$$
\tr W(r,\La) \, = \, \tr \lk -\frac{1}{r^2} \frac{d^2}{d\varphi^2} -   \La - \frac{1}{4r^2} \rk_- \,= \, \sum_{k=1}^{N(r)} \sum_{j \in \N} \lk \La + \frac{1}{4r^2} - \frac{\pi^2 j^2}{r^2 |I_k(r)|^2} \rk_+ \, .
$$
Fix $r > 0$ and $n_0 \in \N$ such that $r_{n_0-1} < r \leq r_{n_0}$. Then the section $\Omega_S(r) \subset [0,2\pi)$ consists of $2^{n_0+1}$ identical open intervals of length $|I_k(r)| = \pi/2^{n_0}$. Hence,
$$
\tr W(r,\La) \, = \, 2^{n_0+1} \sum_{j \in \N} \lk \La + \frac{1}{4r^2} - \frac{2^{2n_0} j^2}{r^2} \rk_+ \, .
$$
Note that for all $j \in \N$ 
$$
\frac{2^{2n_0}j^2}{r^2} - \frac{1}{4r^2} \, \geq \, \frac{2^{2n_0+2}-1}{4r^2} \, \geq \, \frac{15}{4r_1^2} \, .
$$
For $\La \leq 15/4 \cdot r_1^{-2}$ we obtain $\tr W(r,\La) = 0$ and by (\ref{eq:spin:basic}) also $R_\sigma(\La;\Omega_S) = 0$. 

Hence, we can assume $\La > 15/4 \cdot r_1^{-2}$. Suppose that $r > \hat r(\La)$ thus $n_0 > \hat n (\La)$. From (\ref{eq:spin:nm}) we get
$$
\frac{2^{2n_0}j^2}{r^2} - \frac{1}{4r^2} \, \geq \, \frac{2^{2n_0+2}-1}{4r_{n_0}^2} \, \geq \, \La
$$
for all $j \in \N$ and it follows that $\tr W(r,\La) = 0$ for $r > \hat r(\La)$. Moreover, if $r^2 \leq 15/(4\La)$ we have $r \leq r_1$ and 
$$
\frac{4j^2}{r^2} - \frac{1}{4r^2} \, \geq \, \frac{15}{4r^2} \, \geq \La
$$
for all $j \in \N$. Again it follows that $\tr W(r,\La) = 0$ and it remains to consider $\sqrt{15}/(2\sqrt \La) < r < \hat r(\La)$.

For such $r$ we apply (\ref{eq:horn:beta}) to estimate
$$
\tr W(r,\La)^{\sigma+1/2} \, = \, 2^{n_0+1} \sum_{j \in \N} \lk \La + \frac{1}{4r^2} - \frac{2^{2n_0} j^2}{r^2} \rk_+^{\sigma+1/2} \, \leq \, r \lk \La + \frac{1}{4r^2} \rk^{\sigma + 1}  B\lk \frac 12, \sigma + \frac 32 \rk \, .
$$
From (\ref{eq:spin:basic}) we conclude
\begin{align*}
R_\sigma(\La;\Omega_S) \, &\leq \, L^{cl}_{\sigma,1}  B\lk \frac 12, \sigma + \frac 32 \rk \int_{\sqrt{15}/(2\sqrt \La)}^{\hat r (\La)} r \lk \La + \frac{1}{4r^2} \rk^{\sigma + 1} dr\\
&= \, \frac{1}{16(\sigma+1)} \La^\sigma \int_{15}^{4\La \hat r(\La)^2} \lk 1 + \frac 1s \rk^{\sigma+1} ds \\
&\leq \,  \frac{1}{4(\sigma+1)} \hat r(\La)^2 \, \La^{\sigma+1} \, + \, \frac{16^{\sigma-1}}{15^\sigma}  \La^{\sigma} \, \ln \lk 4 \La \hat r(\La)^2 \rk 
\end{align*}
and the claim of the lemma follows from the identity $4\pi (\sigma+1) L^{cl}_{\sigma,2} = 1$.
\end{proof}

Before we give examples we supplement Lemma \ref{lem:spin:main} with the following  lower bound on $R_\sigma(\La;\Omega_S)$.

\begin{lem}
\label{lem:spin:asympt}
Assume there exists $N_0 \in \N$ such that $r_{n-1} < ( 1 - 2^{-n} ) r_n$ is satisfied for all $n \geq N_0$. Then for $\sigma \geq 0$ there exist positive constants $C$ and $\mu$ independent of $\La$ such that 
$$
R_\sigma(\La;\Omega_S) \, \geq \, C \sum_{n=N_0}^{\hat n(\mu \La)} r_n \lk r_n - r_{n-1} \rk \La^{\sigma +1} 
$$
holds for $\La > 0$ with $\hat n (\mu \La) > N_0$.
\end{lem}

\begin{proof}
For $n \geq N_0$ and $k \in \{1,\dots,2^{n+1}\}$ consider a segment $\Omega_{n,k}\subset \Omega_S$, i.e., a region between $r = r_{n-1}$, $r = r_n$ and two adjacent semi-axes $\Gamma_{n,k}$ and $\Gamma_{n,k+1}$. Note that there are $2^{n+1}$ identical segments $\Omega_{n,k}$.
Let $\tau(n)$ denote the maximal number of disjoint squares $Q_{l_n}$ with side length $l_n = r_n/2^{n+1}$ that can be placed in the interior of $\Omega_{n,k}$. From the definition of $\Omega_S$ it follows that 
$$
\tau(n) \, \geq \, C \, \frac{r_n - r_{n-1}}{l_n} \, , \quad n \geq N_0 \, . \footnote{Here and in the following the letter $C$ denotes various positive constants that are independent of $\La$.}
$$
Hence, the variational principle implies
\begin{equation}
\label{eq:spin:asy}
R_\sigma(\La;\Omega_S) \, \geq \, \sum_{n \geq N_0} 2^{n+1} \, \tau(n) \, R_\sigma(\La;Q_{l_n}) \, \geq \, C \sum_{n \geq N_0} 2^{n+1} \, \frac{r_n - r_{n-1}}{l_n} \, R_\sigma(\La;Q_{l_n}) \, .
\end{equation}

To estimate $R(\La;Q_{l_n})$ from below, we first consider the square $Q_1$ with side length $1$. From Weyl's asymptotic law (\ref{eq:int:sc2}) we know that there are positive constants $C$ and $\Lambda_0$, such that $R_\sigma(\La; Q_1) \, \geq \, C \, \La^{\sigma +1}$ holds for all $\La \geq \La_0$. By scaling, we deduce that 
\begin{equation}
\label{eq:spin:sq}
R_\sigma(\La; Q_{l_n}) \, \geq \, C \, l_n^2 \, \La^{\sigma +1}
\end{equation}
holds for all $\La \geq \La_0 / l_n^2$.

Fix $\La > 0$. From (\ref{eq:spin:nm}) we deduce that 
$$
\frac{\La_0}{l_n^2} \, = \, 4\La_0 \frac{2^{2n}}{r_n^2} \, \leq \, 8 \La_0 \lk \frac{2^{2n}}{r_n^2} - \frac{1}{4r_n^2} \rk \, \leq \, \La
$$
holds if $n \leq \hat n (\La/(8\La_0))$. Denoting $\mu = 1/(8\La_0)$ we find that (\ref{eq:spin:sq}) is valid for all squares $Q_{l_n}$ with $n \leq \hat n(\mu \La)$.

In view of (\ref{eq:spin:asy}) it follows that 
$$
R_\sigma(\La;\Omega_S) \, \geq \, C \sum_{n= N_0}^{\hat n (\mu \La)} 2^{n+1}  \, \frac{r_{n+1} - r_n}{l_n} \, l_n^2 \, \La^{\sigma+1} \, \geq \, C \sum_{n = N_0}^{\hat n (\mu \La)} r_n (r_{n} - r_{n-1}) \, \La^{\sigma+1} 
$$
and the proof is complete.
\end{proof}

Let us state some examples to show that the bounds capture the correct order in $\Lambda$ and that choosing different sequences $(r_n)_{n \in \N}$ leads to different behavior in the semiclassical limit.

\begin{cor}
\label{cor:spin}
Let $\sigma \geq 0$.
\begin{enumerate}
\item Assume $r_n = n$. Then for $0<\La \leq 15/4$ we have $R_\sigma(\La; \Omega_S) = 0$ and for $\La > 15/4$
$$
R_\sigma(\La; \Omega_S) \, \leq \, C_\sigma \, \La^{\sigma+1} (\ln \La)^2 \, .
$$
\item Assume $r_n = 2^{\delta n}$ with $0< \delta < 1$. Then for $0<\La \leq 15 \cdot 2^{-2(1+\delta)}$ we have $R_\sigma(\La; \Omega_S) = 0$ and for $\La >  15 \cdot 2^{-2(1+\delta)}$
$$
R_\sigma(\La; \Omega_S) \, \leq \, C_{\sigma,\delta} \, \La^{\sigma+1/(1-\delta)} \, .
$$
\end{enumerate}
All bounds capture the correct order in $\La$ as $\La \to \infty$.
\end{cor}

\begin{proof}
To prove the bounds for $\sigma \geq 3/2$  we can apply Lemma \ref{lem:spin:main} and it remains to estimate $\hat r(\Lambda)$. By definition, $\hat r(\Lambda) = r_{\hat n(\Lambda)}$ and by (\ref{eq:spin:nm}) $r_{\hat n (\Lambda)}$ satisfies
$$
\frac{2^{2\hat n(\La)}}{r_{\hat n(\La)}^2} - \frac{1}{4 r_{\hat n(\La)}^2} \, \leq \, \Lambda \, .
$$
It follows that $\hat r (\La) \, \leq \, C \ln \La$ in the case $r_n = n$
and $\hat r(\La) \, \leq \, C_\delta \La^{\delta/(2(1-\delta))}$
in the case $r_n = 2^{\delta n}$ and the bounds for $\sigma \geq 3/2$ follow from Lemma \ref{lem:spin:main}. To deduce the claimed estimates for $0 \leq \sigma < 3/2$ we apply (\ref{eq:ind:lap}) and finally (\ref{eq:ind:al}).

It remains to prove that the estimates are of correct order in $\Lambda$. Note that in the case $r_n = n$ the assumptions of Lemma \ref{lem:spin:asympt} are satisfied with $N_0 = 1$. Hence, we have
$$
\sum_{n = N_0}^{\hat n(\mu \La )} r_n \lk r_n - r_{n-1} \rk \, = \, \sum_{n = 1}^{\hat n( \mu \La )} n \, \geq \, C \hat n(\mu \La)^2 \, = \, C \hat r(\mu \La)^2 \, .
$$
In the case $r_n = 2^{\delta n}$ we find for sufficiently large $\La$ that 
$$
\sum_{n = N_0}^{\hat n(\mu \La )} r_n \lk r_n - r_{n-1} \rk \, = \,  \sum_{n = N_0}^{\hat n( \mu \La )} 2^{\delta n} \lk 2^{\delta n} -2^{\delta (n-1)} \rk \, \geq \, C \sum_{n = N_0}^{\hat n(\mu \La )} 2^{2\delta n} \, \geq \, C 2^{2\delta \hat n(\mu \La)}  \, = \, C \hat r(\mu \La)^2 \, ,
$$
holds. In both cases, we insert this into Lemma \ref{lem:spin:asympt} and get
\begin{equation}
\label{eq:spin:asympt}
R_\sigma(\La;\Omega_S) \, \geq \, C \La^{\sigma+1} \, \hat r(\mu \La)^2 \, .
\end{equation}
For $\La$ large enough the relations (\ref{eq:spin:nm}) imply
$$
\hat r (\mu \La) \, \geq \, C \ln (\mu \La) \, \geq \, C \ln \La
$$
if $r_n = n$
and
$$
\hat r(\mu \La) \, \geq \, C (\mu \La)^{\delta/(2(1-\delta))} \, \geq \,  C \La^{\delta/(2(1-\delta))} 
$$
if $r_n = 2^{\delta n}$. As $\La \to \infty$ we obtain from (\ref{eq:spin:asympt}) that $R_\sigma(\La;\Omega_S) = O(\La^{\sigma +1} (\ln \La)^2)$ in the case $r_n = n$ and $R_\sigma(\La;\Omega_S) = O(\La^{\sigma + 1/(1-\delta)})$ in the case $r_n = 2^{\delta n}$. Thus the bounds on $R_\sigma(\La, \Omega_S)$ show the correct order in $\La$.
\end{proof}

Let us state one more example, where one encounters exponential growth of the eigenvalue means. 

\begin{cor}
Assume $\sigma \geq 3/2$ and $r_n = 2^n / \sqrt n$. Then for $0 < \La < 15/16$ we have $R_\sigma(\La;\Omega_S) = 0$ and for $\La > 15/16$
$$
R_\sigma(\La;\Omega_S) \, \leq \, C_\sigma  \, 2^{2\La} \, \La^\sigma \, .
$$
\end{cor}

This bound follows from Lemma \ref{lem:spin:main} similar as in Corollary \ref{cor:spin}.


\section{Non-constant potentials}
\label{sec:schr}

In this section we consider Schr\"odinger operators $H_\Omega$ with non-constant potentials $V \geq 0$ on open sets $\Omega \subset \R^d$. Since we define $H_\Omega$ with Dirichlet boundary conditions the variational principle implies that the sharp Lieb-Thirring inequality (\ref{eq:int:lt}) holds. In fact, the Dirichlet condition gives rise to an improvement of this bound. In this section we use this to derive sharp Lieb-Thirring inequalities with remainder term.

\subsection{One-dimensional considerations}
\label{sec:1d}

As in Section \ref{sec:const} we can apply Proposition \ref{pro:basic} to reduce the problem to one dimension. However, for non-constant potentials $V$ the trace of the operator-valued potential $W(x',V)$ defined in (\ref{eq:ind:w}) cannot be calculated explicitly. Therefore we need to study the one-dimensional situation in more detail to derive the following improvement of (\ref{eq:int:lt}). 

\begin{thm}
\label{thm:1d} 
Let $I \subset \R$ be an open interval of length $l < \infty$ and assume $\sigma \geq 3/2$ and $V \in L^{\sigma+1/2}(I)$ such that  
$$
A \,=\, l \int_I V(t) \, dt \, < \, \infty \, .
$$
Then for $A \leq  2 \ln 3$ we have $R_\sigma(V;I)  =  0$ and for $A > 2 \ln 3$ 
$$
R_\sigma(V;I)  \, \leq \, L^{cl}_{\sigma,1} \, \int_I V(t)^{\sigma+1/2} \, dt - \left( \frac{2 \left( \int_I V(t) \, dt \right)^{2} }{\exp(A) -1} \right)^{\sigma} \, .
$$
\end{thm}

The remainder of Section \ref{sec:1d} is devoted to the proof of this result. In particular, we study the effect of different boundary conditions on the eigenvalues. First we assume $I = (0,l)$ and $V \in C_0^\infty(I)$. Recall that 
$$
H_I \, = \,  -\frac{d^2}{dt^2} - V 
$$
is defined in $L^2(I)$ as self-adjoint operator generated by the quadratic form
\begin{equation}
\label{eq:1d:dform}
\langle u, H_I u \rangle \, = \, \int \lk |u'(t)|^2 - V(t) |u(t)|^2 \rk dt \, ,
\end{equation}
with form domain $H_0^1(I)$. Moreover, we define the operator
$$
H_{\mathbb{R}} \, = \, -\frac{d^2}{dt^2} - V
$$
in $L^2(\R)$ generated by the form (\ref{eq:1d:dform}) with form domain $H^1(\R)$.

We assume that the negative spectrum of $H_I$ consists of $N$ eigenvalues $\left( -\lambda_k \right)_{k=1}^N$, $N \in \mathbb{N}$, and denote the negative eigenvalues of $H_{\mathbb{R}}$ by $(-\mu_k)_{k=1}^M$. The variational principle implies $M \geq N$ and $-\mu_k \leq -\lambda_k$ for each $k = 1, \dots, N$.

In order to derive relations between the eigenvalues of $H_I$ and $H_\R$ we define
$$
H_I^{(\alpha,\beta)} \, = \, -\frac{d^2}{dt^2} - V \, , \quad 0 \leq \alpha, \beta \leq \frac \pi2 \, ,
$$
as self-adjoint operators generated by the form
$$
\langle u, H_I^{(\alpha,\beta)} u \rangle \, = \, \int|u'(t)|^2 dt - \int V(t) |u(t)|^2 dt + (\cot \alpha) \, |u(0)|^2 + (\cot \beta) \, |u(l)|^2 
$$
with form domain $H^1(I)$. Note that eigenfunctions of $H_I^{(\alpha,\beta)}$ satisfy boundary conditions of the third kind: $u'(0)  =  (\cot \alpha)  u(0)$ and $u'(l)  =  - (\cot \beta) u(l)$.
For $\alpha, \beta \in \left[ 0,\frac{\pi}{2} \right]$ the negative spectrum of $H_I^{(\alpha,\beta)}$ consists of eigenvalues  $(-\nu_k(\alpha,\beta))_{k=1}^{N(\alpha,\beta)}$. We point out that for $\alpha = \beta = 0$ we recover Dirichlet boundary conditions:
\begin{equation}
\label{eq:1d:dir}
H_I^{(0,0)} \, = \, H_I \, , \quad N(0,0) \, = \, N \, ,  \quad  \textnormal{and} \quad (\nu_k(0,0))_{k=1}^{N(0,0)} \, = \, (\lambda_k)_{k=1}^N \, .
\end{equation}

We need the following result from \cite{Weidma03} about the behavior of the eigenvalues of $H^{(\alpha,\beta)}_I$.
For $\alpha \in \left[ 0,\frac{\pi}{2} \right]$ and $\nu > 0$ let $u(t;\nu,\alpha)$ to be the unique solution of
\begin{align} 
\nonumber
-u''(t) -V(t) u(t) \, &= \, -\nu \, u(t) \, , \quad t \in I \, , \\
\nonumber
u(0;\nu,\alpha) \, &= \, \sin \alpha \, ,\\
\label{eq:1d:eigprob}
u'(0;\nu,\alpha) \, &= \, \cos \alpha \, .
\end{align}

\begin{lem}
\label{lem:pru}
Fix $\beta \in \left[ 0, \frac{\pi}{2} \right]$. Then for $\alpha \in (0, \frac{\pi}{2} )$ the map $\alpha \mapsto \nu_k(\alpha,\beta)$ is monotone increasing and differentiable and we have
$$
\frac{d \nu_k(\alpha,\beta)}{d \alpha} \, = \,  \| u(\cdot;\nu_k(\alpha,\beta),\alpha) \|^{-2}_{L^2(I)} \, .
$$
\end{lem}

Because of the symmetry of the eigenvalue problem (\ref{eq:1d:eigprob}) a corresponding result holds for fixed $\alpha \in \left[ 0,\frac{\pi}{2} \right]$ and the map $\beta \mapsto \nu_k(\alpha,\beta)$, $\beta \in \left[ 0,\frac{\pi}{2} \right]$. For $k = 1, \dots N$ it follows that
$$
-\nu_k(\alpha,\alpha)  \leq  -\nu_k(0,0) =  -\lambda_k  <  0
$$
for all $\alpha \in \left[ 0,\frac{\pi}{2} \right]$.

For $k= 1, \dots N$ put
\begin{equation}
\label{eq:1d:winkel}
\omega_k \, = \, \textnormal{arccot} \sqrt{\mu_k} \ \in \ \left[ 0,\frac{\pi}{2} \right].
\end{equation}
Then we have $N(\omega_k,\omega_k) \geq N$ and both $-\mu_k$ and $-\nu_k(\omega_k,\omega_k)$ exist as negative eigenvalues of $H_\mathbb{R}$ and $H_I^{(\omega_k,\omega_k)}$ respectively.

\begin{pro}
\label{pro:ewerte}
For $k= 1, \dots, N$ the eigenvalues of $H_\mathbb{R}$ and $H_I^{(\omega_k,\omega_k)}$ satisfy 
$$
-\mu_k \, = \, -\nu_k(\omega_k,\omega_k) \, .
$$
\end{pro}

\begin{proof}
For arbitrary $k \in \{1,\dots,N\}$ let $\Phi_k$ denote the eigenfunction of $H_\mathbb{R}$ corresponding to $-\mu_k$. Then $\textnormal{supp} \ V \subset I = (0,l)$ implies
\begin{align*}
\Phi_k(t) \, &= \, c_1 \, \exp \lk -\sqrt{\mu_k}t \rk \quad \textnormal{for} \  t \geq l \quad \textnormal{and} \\
\Phi_k(t) \, &= \, c_2 \, \exp \lk +\sqrt{\mu_k}t \rk \quad \textnormal{for} \ t \leq 0
\end{align*}
with suitable constants $c_1,c_2 \in \mathbb{R}$. From (\ref{eq:1d:winkel}) it follows that $\Phi_k'(0) = (\cot \omega_k) \Phi_k(0)$ and $\Phi_k '(l) = - (\cot \omega_k) \Phi_k(l)$.
Put $\tilde{\Phi}_k = \Phi_k |_{(0,l)}$. Since $\tilde{\Phi}_k$ belongs to the domain of $H_I^{(\omega_k,\omega_k)}$ we find that $-\mu_k$ is an eigenvalue of $H_I^{(\omega_k,\omega_k)}$.
Note that $\Phi_k$ has $k-1$ zeros in the interior of $I$. Therefore $\tilde{\Phi}_k$ has $k-1$ zeros as well and we conclude $-\mu_k  =  -\nu_k(\omega_k,\omega_k)$.
\end{proof}

Similar as in (\ref{eq:1d:eigprob}) let $\tilde{u}(t;\nu,\beta)$, $\beta \in \left[ 0,\frac{\pi}{2} \right]$, $\nu > 0$, be the unique solution of 
\begin{align*} 
- \tilde u''(t) -V(t) \tilde u(t) \, &= \, -\nu \, \tilde u(t) \, , \quad t \in I \, , \\
\tilde u(l;\nu,\beta) \, &= \, \sin \beta \, ,\\
\tilde u'(l;\nu,\beta) \, &= \, -\cos \beta \, .
\end{align*}
Due to the symmetry of the eigenvalue problem (\ref{eq:1d:eigprob}) there is a result analogous to Lemma \ref{lem:pru} relating the derivative of the  map $\beta \mapsto \nu_k(\alpha,\beta)$ to the $L^2$-norm of $\tilde u(\cdot;\nu_k(\alpha,\beta),\beta)$.

In view of (\ref{eq:1d:dir}) and Proposition \ref{pro:ewerte} we have 
$$
\mu_k - \lambda_k \, = \, \nu_k\left( \omega_k,\omega_k \right) - \nu_k(0,0) \, = \, \nu_k (\omega_k,\omega_k) - \nu_k(0,\omega_k) + \nu_k (0,\omega_k) - \nu_k (0,0) \, .
$$
Hence, applying Lemma \ref{lem:pru} and its analog for the map $\beta \mapsto \nu_k(0,\beta)$ yields 
\begin{equation}
\label{eq:1d:diff}
\mu_k - \lambda_k \, = \, \int_0^{\omega_k} \| u(\cdot;{\nu_k(\alpha,\omega_k),\alpha}) \|^{-2}_{L^2(I)} \, d\alpha + \int_0^{\omega_k} \| \tilde{u}(\cdot,{\nu_k(0,\beta),\beta}) \|^{-2}_{L^2(I)} \, d\beta
\end{equation}
for $k=1,\dots,N$.

In the remainder of this subsection we use this identity to complete the proof of Theorem \ref{thm:1d}. In order to get a result valid without further assumptions on the potential $V$ we have to restrict ourselves to considering the ground states.

\begin{lem}
\label{lem:diff}
Let $I \subset \mathbb{R}$ be an open interval of length $l$ and $V \in C_0^\infty(I)$. Then the inequality
$$
\mu_1 - \lambda_1 \, \geq \, \frac{2 \left( \int V(t)  dt \right)^2 }{\exp \left( l \int V(t) dt \right) -1 } 
$$
holds. Moreover, if $l \int V(t) \, dt \leq 2 \ln 3$ then $-\lambda_1 \geq 0$ and we have $R_\sigma(V;I) = 0$ for $\sigma \geq 0$.
\end{lem}

\begin{proof}
First we remark that it suffices to prove the result for $I = (0,l)$. To apply (\ref{eq:1d:diff}) we have to analyze the functions  $u(\cdot;\nu_1(\alpha,\omega_1),\alpha)$ and $\tilde u (\cdot;\nu_1(0,\beta),\beta)$ for $0<\alpha,\beta < \omega_1$.

By definition,  the function $u$ is the first eigenfunction of $H_I^{(\alpha,\omega_1)}$ thus it is non-negative on $I$.
As a solution of (\ref{eq:1d:eigprob}) $u$ solves the integral equation
\begin{eqnarray}
\nonumber
u(t;\nu,\alpha) & = & \frac{1}{2} \lk \sin \alpha \rk \left( e^{\sqrt{\nu}t} + e^{-\sqrt{\nu}t} \right) + \frac 12 \lk \cos \alpha \rk \left( \frac{e^{\sqrt{\nu}t} - e^{-\sqrt{\nu}t}}{\sqrt{\nu}} \right)\\
\label{eq:1d:int}
&& - \int_0^t \frac{\sinh \left( \sqrt{\nu} (t-s) \right)}{\sqrt{\nu}} \, V(s) \, u(s;\nu,\alpha) \, ds \, .
\end{eqnarray}
The first two summands are non-decreasing in $\nu > 0$. For $\alpha \in [0,\omega_1]$, Lemma \ref{lem:pru} and Proposition \ref{pro:ewerte} imply $\nu_1(\alpha,\omega_1) \leq \mu_1$. Since the integrand in (\ref{eq:1d:int}) is positive it follows that 
\begin{eqnarray*}
u(t;\nu_1(\alpha,\omega_1),\alpha) & \leq & \frac{1}{2} \lk \sin \alpha \rk \left( e^{\sqrt{\mu_1}t} + e^{-\sqrt{\mu_1}t} \right) + \frac 12 \lk \cos \alpha \rk \left( \frac{e^{\sqrt{\mu_1}t} - e^{-\sqrt{\mu_1}t}}{\sqrt{\mu_1}} \right) \\
& = & \frac{1}{2} e^{\sqrt{\mu_1}t} \left(\sin \alpha + \frac{\cos \alpha}{\sqrt{\mu_1}} \right) + \frac 12 e^{-\sqrt{\mu_1}t} \left( \sin \alpha - \frac{\cos \alpha}{\sqrt{\mu_1}} \right) \, .
\end{eqnarray*}
Now we use that $\sin \alpha - \cos \alpha / \sqrt{\mu_1} \leq 0$ holds for $\alpha \in [0,\omega_1]$ and conclude
$$
0 \, < \, u(t;\nu_1(\alpha,\omega_1),\alpha) \, \leq \, \frac{1}{2} e^{\sqrt{\mu_1}t} \left( \sin \alpha + \frac{\cos \alpha}{\sqrt{ \mu_1}} \right)\, .
$$
By explicit calculations it follows that
$$
\int_0^{\omega_1} \| u(\cdot;\nu_1(\alpha,\omega_1),\alpha) \|^{-2}  \, d\alpha \, \geq \, \frac{4\mu_1}{\exp \left( 2l \sqrt{\mu_1} \right) -1} \, .
$$
Similarly, we find
$$
\int_0^{\omega_1} \| \tilde u(\cdot;\nu_1(0,\beta),\beta) \|^{-2}  \, d\beta \, \geq \, \frac{4\mu_1}{\exp \left( 2l \sqrt{\mu_1} \right) -1} 
$$
and (\ref{eq:1d:diff}) implies
\begin{equation}
\label{eq:1d:gdstdiff}
\mu_1- \lambda_1 \, \geq \,  \frac{8\mu_1}{\exp \left( 2l \sqrt{\mu_1} \right) -1} \, .
\end{equation}
For $l \sqrt {\mu_1} \leq \ln 3$ it follows that $-\lambda_1 \geq 0$.
Since the right hand side of (\ref{eq:1d:gdstdiff}) is non-increasing the estimate \cite{HuLiTh98}
$$
\sqrt{\mu_1} \, \leq \, \frac 12 \int_I V(t) \, dt 
$$
implies the claimed result.
\end{proof}

The proof of Theorem \ref{thm:1d} is an immediate consequence of the results above:

\begin{proof}[Proof of Theorem \ref{thm:1d}]
Using convexity of the map $\lambda \mapsto \lambda^\sigma$ and the Lieb-Thirring inequality (\ref{eq:int:lt}) we estimate
$$
R_\sigma(V;I) \, = \, \sum_{k=1}^N \lambda_k^\sigma \, \leq \,  \sum_{k=1}^N \mu_k^\sigma - \left( \mu_1^\sigma - \lambda_1^\sigma \right) \, \leq \, L^{cl}_{\sigma,1} \int_I V(t)^{\sigma+1/2} \, dt - \left( \mu_1 - \lambda_1 \right)^\sigma \, .
$$
Hence, for $V \in C_0^\infty(I)$ the claim follows from Lemma \ref{lem:diff}. A standard approximation argument allows us to prove the claim for all non-negative potentials $V \in L^{\sigma+1/2}(I)$. 
\end{proof}

\subsection{A sharp Lieb-Thirring inequality with remainder term}

Let us now consider general Schr\"{o}digner operators $H_\Omega$ on bounded or quasi-bounded open sets $\Omega \subset \mathbb{R}^d$ with Dirichlet boundary conditions. To apply the inductive argument introduced in Section \ref{sec:ind}, fix  a coordinate system in $\R^d$.  For $x \in \Omega$ we write $x = (x',t) \in \R^{d-1} \times \R$ and assume that $V_{x'}  \in L^{\sigma+d/2}(\Omega(x'))$, a.e. in $x' \in \R^{d-1}$. We use the notation introduced in Section \ref{sec:ind} and put
\begin{align*}
A_k(x') \, &= \, \left| J_k(x') \right| \int_{J_k(x')} V_{x'}(t) \, dt \, ,\\
B_k(x') \, &= \, \int_{J_k(x')} V_{x'}(t) \, dt \, .
\end{align*}
Let $\kappa (x',V) \subset \mathbb{N}$ be the subset of all indices $k$ with $A_k(x') > 2 \ln 3$ and put
$$ \Omega_V(x') \, = \, \bigcup_{k \in \kappa(x',V) } J_k(x') \subset \mathbb{R} \quad \mbox{and} \quad \Omega_V \, = \, \bigcup_{x' \in \mathbb{R}^{d-1}} \{ x' \} \times \Omega_V(x') \subset \Omega \, .$$
The results from Section \ref{sec:ind} and Section \ref{sec:1d} imply the following sharp Lieb-Thirring inequality with remainder term.

\begin{thm}
\label{thm:ltgen}
Let $\Omega$ be an open set in $\R^d$, $d \geq 2$, and assume $\sigma \geq 3/2$. Then the estimate  
$$
R_\sigma(V;\Omega) \, \leq \, L_{\sigma,d}^{cl}  \int_{\Omega_V} V(x)^{\sigma+d/2} \, dx  -  L^{cl}_{\sigma,d-1} \int_{\R^{d-1}} \rho(x',V) \, dx'
$$
holds with a remainder
$$
\rho(x',V) \, = \,  \sum_{k \in \kappa(x',V)} \left( \frac{2 B_k(x')^2}{\exp \left( A_k(x') \right)-1} \right)^{\sigma + (d-1)/2}  \, .
$$
\end{thm}

\begin{proof}
In view of Proposition \ref{pro:basic} we have to estimate
$$
\tr W(x',V)^{\sigma + (d-1)/2} \, = \, \sum_{k=1}^{N(x')} \mbox{Tr}  H_k(x')_-^{\sigma + (d-1)/2} \, = \,  \sum_{k=1}^{N(x')} R_{\sigma + (d-1)/2}(V_{x'};J_k(x')) \, .
$$
The potential $V_{x'}$ satisfies the conditions of Theorem  \ref{thm:1d}, a.e. in $x' \in \R^{d-1}$. For $k \notin \kappa(x',V)$ we have $|J_k(x')| \int_{J_k(x')} V_{x'}  dt \leq 2 \ln 3$ and Theorem \ref{thm:1d} yields $\mbox{Tr}  H_k(x')_- = 0$. Hence,
\begin{align*}
\lefteqn{ \textnormal{Tr} W(x',V)^{\sigma+(d-1)/2} }\\
& =  \sum_{k \in  \kappa(x',V)} R_{\sigma + (d-1)/2}(V_{x'};J_k(x'))\\
& \leq  \sum_{k \in \kappa(x',V)} \left( L^{cl}_{\sigma+(d-1)/2,1} \int_{J_k(x')} V_{x'}(t)^{\sigma+d/2}  dt - \left( \frac{2 B_k(x')^2}{\exp \left( A_k(x') \right)-1} \right)^{\sigma + (d-1)/2} \right).
\end{align*}
Thus the claim follows from Proposition \ref{pro:basic} using the identities
$$
\int_{\R^{d-1}} \sum_{k \in \kappa(x',V)} \int_{J_k(x')} V_{x'}(t)^{\sigma+d/2} dt \, dx' \, = \, \int_{\Omega_V} V(x)^{\sigma+d/2} \,  dx
$$
and $L^{cl}_{\sigma,d-1} \, L^{cl}_{\sigma+(d-1)/2,1} = L^{cl}_{\sigma,d}$.
\end{proof}

\subsection{An example with $V \notin L^{\sigma+d/2}$}

Let us illustrate Theorem \ref{thm:ltgen} by an example of a Schr\"odinger operator defined on a horn-shaped region with a potential such that the classical Lieb-Thirring inequality (\ref{eq:int:lt}) fails. As in Section \ref{sec:horn} set
$$
\Omega_1 \, = \, \left\{ (x,y) \in \R^2 \, : \, |x| \cdot |y| \, \leq \, 1 \right\}  
$$
and put $V_\alpha(x,y) =  |x|^\alpha |y|^{-\alpha}$ with $0 < \alpha < 2/5$. Again, we introduce a scaling parameter $\lambda > 0$ and study the operator 
$$
H_\alpha \, = \, -\Delta - \lambda V_\alpha \, ,
$$
defined in $L^2(\Omega_1)$ with Dirichlet boundary conditions. Since $V_\alpha \notin L^{\sigma+1}(\Omega_1)$ the classical results (\ref{eq:int:lt}) and (\ref{eq:int:sc1}) fail. 

Nevertheless, Theorem \ref{thm:ltgen} yields an upper bound on $R_\sigma(\lambda V_\alpha;\Omega_1)$ for $3/2 \leq \sigma < (1-\alpha)/\alpha$. Indeed, for any $x \in \R$ the section $\Omega_1(x)$ consists of one open interval $(-x^{-1},x^{-1})$ and
$$
A_1(x) \, = \, \frac{4}{|x|} \int_0^{|x|^{-1}} \lambda |x|^\alpha |y|^{-\alpha} \, dy \, = \, \frac{4\lambda}{1-\alpha} |x|^{2(\alpha-1)} \, .
$$
Since $\alpha < 1$ we find that $A_1(x)$ tends to zero as $|x|$ tends to infinity. Thus $A_1(x) \leq 2 \ln 3$ holds for
$$
|x| \geq \lk \frac{2\lambda}{(1-\alpha ) \ln 3} \rk^{1/(2-2\alpha)} \, = \, x_\alpha(\lambda) \, .
$$
From Theorem \ref{thm:ltgen} it follows that for all $3/2 \leq \sigma < (1-\alpha)/\alpha$ the estimate
\begin{align*}
R_\sigma(\lambda V_\alpha;\Omega_1) \, &\leq \, 4 L^{cl}_{\sigma,2} \int_0^{x_\alpha(\lambda)} \int_0^{x^{-1}} x^{\alpha(\sigma+1)} \, y^{-\alpha(\sigma+1)} \, dy \, dx \, \lambda^{\sigma+1}\\
&\leq \, L^{cl}_{\sigma,2}  \, \frac{4}{2\alpha(\sigma+1)(1-\alpha (\sigma + 1))} \lk \frac{2}{(1-\alpha)\ln 3} \rk^{\alpha(\sigma+1) /(1-\alpha )} \, \lambda^{(\sigma+1)/(1-\alpha)}
\end{align*}
holds for all $\lambda > 0$.


\begin{thebibliography}{KVW09}

\bibitem[Ada70]{Adams70}
R.~A. Adams, \emph{Capacity and compact imbeddings}, J. Math. Mech. \textbf{19}
  (1970), 923--929.

\bibitem[AF03]{AdaFou03}
R.~A. Adams and J.~F. Fournier, \emph{Sobolev spaces}, second ed., Academic
  Press, 2003.

\bibitem[AL78]{AizLie78}
M.~Aizenman and E.H. Lieb, \emph{On semi-classical bounds for eigenvalues of
  {S}chr{\"o}dinger operators}, Phys. Lett. \textbf{66} (1978), 427--429.

\bibitem[Ber72]{Berezi72b}
F.~A. Berezin, \emph{Covariant and contravariant symbols of operators}, Izv.
  Akad. Nauk SSSR Ser. Mat. \textbf{13} (1972), 1134--1167.

\bibitem[BS87]{BirSol87}
M.~Sh. Birman and M.~Z. Solomjak, \emph{Spectral theory of selfadjoint
  operators in {H}ilbert space}, Mathematics and its Applications (Soviet
  Series), D. Reidel Publishing Co., Dordrecht, 1987.

\bibitem[CH24]{CouHil24}
R.~Courant and D.~Hilbert, \emph{Methoden der mathematischen {P}hysik},
  Springer, Berlin, 1924.

\bibitem[Cla67]{Clark67}
C.~Clark, \emph{{R}ellich's embedding theorem for a 'spiny urchin'}, Canad.
  Math. Bull \textbf{10} (1967), 731--734.

\bibitem[DS92]{DavSim92}
E.~B. Davies and B.~Simon, \emph{Spectral properties of the {N}eumann
  {L}aplacian of horns}, Geom. Funct. Anal. \textbf{2} (1992), 105--117.

\bibitem[FG10]{FraGei10}
R.~L. Frank and L.~Geisinger, \emph{Two-term spcetral asymptotics of the
  {D}irichlet {L}aplacian on a bounded domain}, to appear in the proceedings of
  QMath 11, World Scientific, Singapore (2010).

\bibitem[Fle78]{Flecki78}
J.~Fleckinger, \emph{Répartition des valeurs propres d'opérateurs elliptiques
  sur des ouverts non bornés}, C. R. Acad. Sci. Paris Sér. A \textbf{286}
  (1978), no.~3, 149--152.

\bibitem[FLU02]{FrLiUe02}
J.~K. Freericks, E.~H. Lieb, and D.~Ueltschi, \emph{Segregation in the
  {F}alicov-{K}imball model}, Comm. Math. Phys. \textbf{227} (2002), no.~2,
  243--279.

\bibitem[GLW11]{GeLaWe11}
L.~Geisinger, A.~Laptev, and T.~Weidl, \emph{Geometrical versions of improved
  {B}erezin-{L}i-{Y}au inequalities}, to appear in the Journal of Spectral
  Theory (2011).

\bibitem[GW10]{GeiWei10}
L.~Geisinger and T.~Weidl, \emph{Universal bounds for traces of the {D}irichlet
  {L}aplace operator}, J. Lond. Math. Soc. \textbf{82} (2010), no.~2, 395--419.

\bibitem[HLT98]{HuLiTh98}
D.~Hundertmark, E.~H. Lieb, and L.E. Thomas, \emph{A sharp bound for an
  eigenvalue moment of the one-dimensional schr{\"o}dinger operator}, Adv.
  Theor. Math. Phys. \textbf{2} (1998), 719--731.

\bibitem[H{\"o}r85]{Hoerma85}
L.~H{\"o}rmander, \emph{The analysis of linear partial differential operators,
  vol. 4}, Springer-Verlag, Berlin, 1985.

\bibitem[Ivr80]{Ivrii80}
V.~Y. Ivrii, \emph{On the second term of the spectral asymptotics for the
  {L}aplace-{B}eltrami operator on manifolds with boundary}, Funtsional. Anal.
  i Prilozhen. \textbf{14} (1980), no.~2, 25--34.

\bibitem[Ivr98]{Ivrii98}
\bysame, \emph{Microlocal analysis and precise spectral asymptotics}, Springer
  Monographs in Mathematics, Springer-Verlag, Berlin, 1998.

\bibitem[KVW09]{KoVuWe09}
H.~Kova\v{r}\'{\i}k, S.~Vugalter, and T.~Weidl, \emph{Two dimensional
  {B}erezin-{L}i-{Y}au inequalities with a correction term}, Comm. Math. Phys.
  \textbf{287} (2009), no.~3, 959--981.

\bibitem[Lap97]{Laptev97}
A.~Laptev, \emph{Dirichlet and {N}eumann eigenvalue problems on domains in
  {E}uclidean spaces}, J. Funct. Anal. \textbf{151} (1997), no.~2, 531--545.

\bibitem[Lie73]{Lieb73}
E.~H. Lieb, \emph{The classical limit of quantum spin systems}, Comm. Math.
  Phys. \textbf{31} (1973), 327--340.

\bibitem[Lie97]{Lieb97}
\bysame, \emph{The stability of {M}atter: {F}rom {A}toms to {S}tars, {S}electa
  of {E}lliot {H}. {L}ieb, ed. by {W}. {T}hirring}, Springer, 1997.

\bibitem[LS10]{LieSei10}
E.~H. Lieb and R.~Seiringer, \emph{The stability of matter in quantum
  mechanics}, Cambridge University Press, Cambridge, 2010.

\bibitem[LT76]{LieThi76}
E.~H. Lieb and W.~Thirring, \emph{Inequalities for the moments of the
  eigenvalues of the {S}chr{\"o}dinger hamiltonian and their relation to
  {S}obolev inequalities}, Studies in Math. Phys., Essays in Honor of
  {V}alentine {B}argmann (E.~Lieb, B.~Simon, and A.~S. Wightman, eds.),
  Princeton Univ. Press, Princeton, New Jersey, 1976, pp.~269--330.

\bibitem[Lun10]{Lundho10}
D.~Lundholm, \emph{Weighted supermembrane toy model}, Lett. Math. Phys.
  \textbf{92} (2010), no.~2, 125--141.

\bibitem[LW00]{LapWei00}
A.~Laptev and T.~Weidl, \emph{Sharp {L}ieb-{T}hirring inequalities in high
  dimensions}, Acta Math. \textbf{184} (2000), no.~1, 87--111.

\bibitem[LY83]{LiYau83}
P.~Li and S.~T. Yau, \emph{On the {S}chr{\"o}dinger equation and the eigenvalue
  problem}, Comm. Math. Phys. \textbf{88} (1983), no.~3, 309--318.

\bibitem[Mel03]{Melas03}
A.D. Mel\'{a}s, \emph{A lower bound for sums of eigenvalues of the
  {L}aplacian}, Amer. Math. Soc \textbf{131} (2003), 631--636.

\bibitem[MM06]{MatMue06}
S.~G. Matinyan and B.~M{\"u}ller, \emph{Adventures of the coupled
  {Y}ang-{M}ills oscillators. {I}. {S}emiclassical expansion}, J. Phys. A
  \textbf{39} (2006), no.~1, 45--59.

\bibitem[RS78]{ReeSim78}
M.~Reed and B.~Simon, \emph{Methods of {M}odern {M}athematical {P}hysics
  {I}{V}. {A}nalysis of {O}perators}, Academic Press, 1978.

\bibitem[Sim83]{Simon83}
B.~Simon, \emph{Non-classical eigenvalue asymptotics}, J. Functional Anal.
  \textbf{53} (1983), 84--98.

\bibitem[SV97]{SafVas97}
Y.~Safarov and D.~Vassiliev, \emph{The asymptotic distribution of eigenvalues
  of partial differential operators}, Translations of Mathematical Monographs,
  155, American Mathematical Society, Providence, RI, 1997.

\bibitem[vdB84]{Berg84b}
M.~van~den Berg, \emph{On the spectrum of the {D}irichlet {L}aplacian for
  horn-shaped regions in $\mathbb{R}^n$ with infinite volume}, J. Funct. Anal.
  \textbf{58} (1984), 150--156.

\bibitem[vdB92a]{Berg92a}
\bysame, \emph{{D}irichlet-{N}eumann bracketing for horn-shaped regions}, J.
  Funct. Anal. \textbf{104} (1992), no.~1, 110--120.

\bibitem[vdB92b]{Berg92b}
\bysame, \emph{On the spectral counting function for the {D}irichlet
  {L}aplacian}, J. Funct. Anal. \textbf{107} (1992), no.~2, 352--361.

\bibitem[Wei03]{Weidma03}
J.~Weidmann, \emph{Lineare {O}peratoren in {H}ilbertr{\"a}umen. {T}eil {I}{I}:
  {A}nwendungen}, B.G. Teubner, Stuttgart, 2003.

\bibitem[Wei08]{Weidl08}
T.~Weidl, \emph{Improved {B}erezin-{L}i-{Y}au inequalities with a remainder
  term}, Amer. Math. Soc. Transl. \textbf{225} (2008), no.~2, 253--263.

\bibitem[Wey12]{Weyl12a}
H.~Weyl, \emph{Das asymptotische {V}erteilungsgesetz der {E}igenwerte linearer
  partieller {D}iffenertialgleichungen (mit einer {A}nwendung auf die {T}heorie
  der {H}ohlraumstrahlung)}, Math. Ann. \textbf{71} (1912), no.~4, 441--479.

\end{thebibliography}
\end{document}